\documentclass[12pt,a4paper,twoside]{amsart}
\usepackage{amsfonts, amsthm, amsmath, amssymb}
\usepackage{hyperref}
\usepackage{mathrsfs,amsmath}
\hypersetup{colorlinks=false}

\usepackage[margin=1.25in]{geometry}

\usepackage{helvet}
\usepackage{hyperref}

\usepackage{amsthm}
\newtheorem{theorem}{Theorem}
\newtheorem{lemma}{Lemma}[section]
\newtheorem{remark}{Remark}

\newtheorem{proposition}{Proposition}

\begin{document}

\author{Aritra Ghosh and Kummari Mallesham}
\title{Sub-Weyl strength bounds for twisted $GL(2)$ short character sums}

\address{ Aritra Ghosh and Kummari Mallesham \newline {\em Stat-Math Unit, Indian Statistical Institute, 203 B.T. Road, Kolkata 700108, India; \newline  Email:  aritrajp30@gmail.com
} }

\address{Kummari Mallesham \newline {\em Department of Mathematics, IIT Bombay, Powai, India, 400076; \newline  Email:  iitm.mallesham@gmail.com
} }

\maketitle

\begin{abstract}
 Let $$S(N) = \sum_{n \sim N}^{\text{smooth}} \, \lambda_{f}(n) \, \chi(n),$$
 where $\lambda_{f}(n)$'s are Fourier coefficients of Hecke-eigen form, and $\chi$ is a primitive character of conductor $p^{r}$. In this article we prove a sub-Weyl strength bounds for $S(N)$. Indeed, we obtain
 $$S(N) \ll \, N^{\frac{5}{9}} \ p^{\frac{13r}{45}},$$
 provided that $ p^{13r/20} \leq N \leq p^{4r/5}$.  Note that the above bound for $S(N)$ is non-trivial if $N\geq \left(p^{r}\right)^{\frac{2}{3}-\frac{1}{60}}$. 
\end{abstract}

\tableofcontents

\section{ Introduction } 

Let $f$ be a holomorphic Hecke-eigen cusp form on $SL(2,\mathbb{Z})$ with normalized Fourier coefficients $\lambda_{f}(n)$. Let $\chi$ be a Dirichlet character of conductor $p^{r}$. In this article we consider the character sum
$$S_{f,\chi}(N) = \sum_{n=1}^{\infty} \, \lambda_{f}(n) \, \chi(n) \, W\left(\frac{n}{N}\right),$$
where $W$ is a bump function supported on the interval $[1,2]$ and satisfies $W^{(j)}(x) \ll_{j} 1$.  One can use information (bounds for $L(1/2, f \times \chi)$) about $L$-values $L(1/2, f\times \chi)$ to show cancellations in the sum $S_{f,\chi}(N)$. Indeed, by Mellin inversion we have that

$$S_{f,\chi}(N) = \frac{1}{2 \pi i} \int_{(\sigma)} N^{s} \tilde{W}(s) \, L\left(s, f\times \chi\right) \, ds, \quad \sigma >1.$$
First shift the contour to $1/2$-line and estimate trivially to get
$$S_{f,\chi}(N) \ll N^{1/2} \, \vert L(1/2, f \times \chi ) \vert  N^{\epsilon}, $$
where we have used the fact that $\tilde{W}(s)$ decays rapidly as $\Im(s) \to \infty$. We fix the form $f$. Then the analytic conductor of the $L$ value $$L(1/2, f \times \chi)$$ becomes $p^{2r}$. The convexity bound $L(1/2, f \times \chi) \ll p^{r/2}$ would imply that
$$S_{f,\chi} (N) \ll N^{1/2} \, p^{r/2} N^{\epsilon},$$
which is non-trivial if $N> p^{r}$. The best known bound for $L(1/2, f \times \chi)$ is the Weyl bound $L(1/2, f \times \chi) \ll p^{r/3}$, due to D. Mili\'cevi\'c and V. Blomer \cite{miliblom}, and R. Munshi and S. Singh \cite{munshisingh}, which would then imply that
$$S_{f,\chi} (N) \ll N^{1/2} \, p^{r/3} N^{\epsilon},$$
which is non-trivial if $N> p^{2r/3}$. Currently we do not know how to obtain sub-Weyl bounds for $L(1/2, f \times \chi)$. But the sub-Weyl type bounds $L(1/2, f \times \chi) \ll (p^{r})^{\frac{1}{3}-\eta}$ for some $\eta >0$ would give non-trivial bounds for $S_{f},\chi(N)$ whenever $N > (p^{r})^{\frac{2}{3}-2 \eta}$. It is needless to mention that Lyndel\"of hypothesis would give non-trivial bounds for $L(1/2, f \times \chi)$ if $N > p^{r\epsilon}$.

Let $K(n)$ be a trace function modulo prime $p$. In \cite{michle}, E. Fouvry, E. Kowalski and P. Michel showed that
$$\sum_{n} \lambda_{f}(n) \, K(n) \, W\left(\frac{n}{N}\right) \ll N^{1/2} p^{3/8} N^{\epsilon}.$$
This is Burgess type bound which gives non-trivial bounds for above sums  if $N> p^{3/4}$. Recently, In \cite{ghosh} the first author showed that we have cancellation  in the above sums  if $N> p^{2/3}$ (Weyl strength) in the case when the trace function  $K(n)$ is taken to be a Dirichlet character. But for general trace functions $K(n)$ the bound of Fouvry, Kowalski, and Michel is the best. 

We are interested in showing that sub-Weyl strength cancellations in $S_{f,\chi}(N)$ when $\chi$ is Dirichlet character of conductor $p^{r}$ and $r \to \infty$ (depth aspect). This will be a counterpart result to the that of R. Holowinsky, R. Munshi, and Z. Qi \cite{RRZ} where they showed sub Weyl strength cancellations in anlytic twist of $\lambda_{f}(n)$.

Our aim in this article is to establish the following theorem.

\begin{theorem} \label{mainth} Let $p$ be odd prime such that $p > 5$. Then we have 
$$S_{f, \chi}(N) \ll  N^{\frac{5}{9}} \ p^{\frac{13r}{45}} \, N^{\epsilon},$$
where implied constant depends on the prime $p$, and provided $p^{13r/20} \leq N\leq p^{4r/5}$.
\end{theorem}
\begin{remark}
We need condition $p > 5$ on the prime $p$ to apply $p$-adic exponent pair $(1/30,13/15)$.
\end{remark}
\begin{remark}
we can get same kind of bounds for $S_{f},\chi(N)$ even if we consider Fourier coefficients $\lambda_{f}(n)$ of Hecke Maass cusp form $f$ as we only require Ramanujan bound for the coefficients in $L^{2}$ sense.  
\end{remark}

\subsection{Method of the proof}
We take the path of circle method to bound the sum $S_{f,\chi}(n)$, especially the approach of R. Munshi. First we separate the oscillations $\lambda_{f}(n)$ and $\chi(n)$ using the delta symbol. While separating these oscillations we introduce extra additive harmonic in the sum which serve as conductor lowering in the delta method. Once these get separated we apply Voronoi and Poisson summation formulas accordingly. We then remove the Fourier coefficients by applying Cauchy-Schwartz inequality. In the resulting expression we open the absolute square and again we employ Poisson summation. In zero frequency we can not do much other than evaluating trivially. But in non-zero frequency we end up with sums of the form 
$$\sum_{ R \leq m \leq 2R} e\left(\frac{f(m)}{p^{r}}\right)$$
with $R \leq p^{r}/N$ and ``nice" phase function $f$. In this sum we seek to get some cancellations which we can achieve by urging $p$-adic exponent pair $(1/15, 13/15)$. The main novelty of this paper is to apply $p$-adic exponent pair to get better bounds. These $p$-adic analogue exponent pair is developed by D. Mili\'cevi\'c in \cite{mil}.

\subsection{Notations}
We write $p^{s} \parallel m$ to denote that $p^{s}\mid m$ and $p^{s+1} \nmid m$.

\section{An application of the circle method} \label{circlemethod}

We separate the oscillations $\lambda_{f}(n)$ and $\chi(n)$ in the sum $S_{f,\chi}(N)$ by  using the delta symbol $\delta$ which is defined on the set of integers by $\delta(0)=1$ and $\delta(m) =0$ if $m\neq 0$.  We have the following expression for $\delta$ which is due to Duke, Friedlander and Iwaniec  \cite{DFI}. Let $L\geq 1$ be a large number. For $n \in [-2L,2L]$, we have

\begin{align*} 
\delta(n)= \frac{1}{Q} \sum_{1 \leq q \leq Q} \frac{1}{q} \, \sideset{}{^\star}{\sum}_{a \, \rm mod \, q} \, e \left(\frac{na}{q}\right) \int_{\mathbb{R}} g(q,x) \,  e\left(\frac{nx}{qQ}\right) \, \mathrm{d}x,
\end{align*}
where  $Q=2L^{1/2}$, and the function  $g(q,x)$ satisfies the following properties (see, \cite[Lemma 5,]{Haung}): 
\begin{itemize}
	\item $g(q,x) = 1 +O\left( \frac{Q}{q}\left(\frac{q}{Q}+|x|\right)^{A}\right), \quad g(q,x) \ll |x|^{-A}$ for any $A>1$.
	\item $x^{j} \frac{\partial^{j}}{\partial x^{j}} g(q,x) \ll \log Q \,  \min \{ \frac{Q}{q},\frac{1}{|x|}\}$.
	\item $\int_{\mathbb{R}} \left(|g(q,x)|+|g(q,x)|^{2}\right) \rm dx \ll Q^{\epsilon}$.
\end{itemize}

Indeed,  we have

$$S_{f,\chi}(N) = \mathop{\sum \sum}_{ \substack {{m, n=1} \\  p^{\ell} \mid (n-m) }}^\infty \lambda_f (n)  \chi(m) \, \delta \left( \frac{n -m}{p^{\ell}} \right) W\left( \frac{n}{N} \right) V\left( \frac{m}{N} \right)$$
with the condition that \underline{$p^{\ell} \leq N$ and $\ell \leq r$}. Now by writing the expression for $\delta$, with the choice $Q =\sqrt{N/p^{\ell}}$, in the above sum we arrive at

\begin{align} 
S_{f,\chi}(N)=   \frac{1}{Qp^{\ell}} \int_{\mathbb{R}}  \sum_{1 \leq q\leq Q} \frac{g(q,x)}{q} 
\mathop{ \sideset{}{^\star}\sum}_{a \, \mathrm{mod} \, q}    \, \sum_{b \,  \mathrm{mod} \, p^{\ell}}  \, \mathcal{S}_{f}(N;a,b,q,x) \,\mathcal{S}_{\chi}(N;a,b,q,x) \,  dx,  
\end{align} 
where
\begin{equation} \label{nsum}
\mathcal{S}_{f}(N;a,b,q,x) = \sum_{n=1}^\infty  \lambda_f(n)   e\left(  \frac{ (a+ bq) n}{p^{\ell} q} \right) e\left(  \frac{ x n}{  p^{\ell}  qQ} \right)   W\left( \frac{n}{N} \right) ,
\end{equation}
and 
\begin{equation} \label{msum}
\mathcal{S}_{\chi}(N;a,b,q,x) = \sum_{m=1}^\infty   \chi(m)  e\left( - \frac{ (a+ bq) m}{p^{\ell} q} \right) e\left(  \frac{ -m x }{  p^{\ell} qQ} \right)  V\left( \frac{m}{N} \right).
\end{equation}

\section{Application of summation formulas}

\subsection{Applying Poisson summation Formula}
We shall apply the Poisson summation formula to the sum over $m$ in equation  \eqref{msum} to get the following lemma.

\begin{lemma} \label{poissonsum}
We have
$$\mathcal{S}_{\chi}(N;a,b,q,x)= \frac{N}{p^{r} q} \, \sum_{\substack{m \in \mathbb{Z} \\ |m| \leq M_{0}}}\, \mathcal{C} (a, b , q, m) \, \mathcal{I} (x , q, m) + O\left(N^{-2022}\right),$$
with $M_{0} := \frac{p^{r}Q}{N} N^{\epsilon}$,
where
$$\mathcal{C} (a, b , q, m) = \sum_{\beta(p^{r} q)} \chi(\beta) e\left( - \frac{ (a+ bq) \beta}{p^{\ell} q}  + \frac{m \beta}{ p^r q}\right) ,$$
and 
$$ \mathcal{I} (x , q, m) = \int_{\mathbb{R}} V (z)  e\left(  \frac{ - N x z}{  p^{\ell}  qQ} \right) e\left(  \frac{ - N m z }{  p^{r} q} \right) dz.$$
\end{lemma}

\begin{proof}
We split the $m$-sum in \eqref{msum} into congruence classes modulo $p^{r}q$. Indeed, we write $m=\beta + c p^{r} q$ with $\beta \, \rm \, mod \, \, p^{r}q$, and $c \in \mathbb{Z}$ to get

\begin{align*}
\mathcal{S}_{\chi}(N;a,b,q,x) &  = \sum_{\beta(p^{r} q)} \chi(\beta) e\left( - \frac{ (a+ bq) \beta}{p^{\ell} q} \right) \sum_{ c \in \mathbb{Z}}  V\left( \frac{ \beta +  c p^{r} q }{N} \right)   e\left(  \frac{ -(\beta +  c p^{r} q) x }{  p^{\ell} qQ} \right) \\
&  = \sum_{\beta(p^{r} q)} \chi(\beta) e\left( - \frac{ (a+ bq) \beta}{p^{\ell} q} \right) \sum_{ m \in \mathbb{Z}} \int_{\mathbb{R}} V\left( \frac{ \beta + y p^{r} q }{N} \right)   e\left(  \frac{ -(\beta + y p^{r} q) x }{  p^{\ell} qQ} \right) e(-my) dy, 
\end{align*} 
the second equality follows by applying Poisson summation formula. We now substitute the change of variable $ ( \beta + y p^{r} q)/N = z $ to obtain the value of $\mathcal{S}_{\chi}(N;a,b,q,x)$ to be
\begin{align*}
& \frac{N}{p^{r} q}\sum_{ m \in \mathbb{Z}} \left\lbrace \sum_{\beta(p^{r} q)} \chi(\beta) e\left( - \frac{ (a+ bq) \beta}{p^{\ell} q}  + \frac{m \beta}{ p^r q}\right) \right\rbrace \int_{\mathbb{R}} V (z)  e\left(  \frac{ - N x z}{  p^{\ell}  qQ} \right) e\left(  \frac{ - N m z }{  p^{r} q} \right) dz   \\
&=  \frac{N}{p^{r} q} \, \sum_{m \in \mathbb{Z}}\, \mathcal{C} (a, b , q,m) \, \mathcal{I} (x , q, m),
\end{align*}
where $\mathcal{C} (a, b , q,m),  \, \mathcal{I} (x , q, m)$ are given as above. We see, by repeated integration by parts, that
 $$\mathcal{I} (x , q, m) \ll_{j} \left(1+\frac{N|x|}{p^{\ell}qQ}\right)^{j} \left(\frac{p^{r}q}{Nm}\right)^{j},$$
 for any $j \geq 0$. Thus, $\mathcal{I} (x , q, m)$ is negligibly small unless 
 $$|m| \leq M_{0} := \frac{p^{r}Q}{N} N^{\epsilon}.$$
\end{proof}

We now first evaluate the character sum in the following subsection.

\subsection{Evaluation of the character sum}
We have the following lemma.

\begin{lemma} \label{charsumlema}
Let $q=p^{r_{1}}q^{\prime}$ with $(p,q^{\prime})=1$ (i.e., $p^{r_{1}} \parallel q$). Then we have
\begin{align*}
\mathcal{C} (a, b , q,m) & = \begin{cases}
q \ \chi (q^{\prime}) \ \overline{\chi} \left( \frac{m- (a + bq ) p^{r- \ell}}{p^{r_1}}\right)  \tau_{\chi}    \ \  \ \ \   \textrm{if }  \ \ \ \   a  \equiv  m  \, \overline{p^{r-\ell}}  \mod  q^{\prime}, \quad \mathrm{and} \quad p^{r_{1}} \parallel m\\ 
0   \ \ \ \     \textrm{otherwise} ,
\end{cases}
\end{align*}
where $ \tau_{\chi} $ denotes the Gauss sum. 
\end{lemma}

\begin{proof}
Since $q= p^{r_1} q^\prime$ with $(p, q^\prime)=1$, the character sum $\mathcal{C} (a, b , q,m)$ is given by 
\begin{align*}
 \sum_{\beta (p^{r+ r_1} q^\prime)} \chi(\beta) e  \left( \frac{  - (a+ bq ) \beta }{ p^{\ell+r_1} q^\prime }  + \frac{ m \beta }{ p^{r+r_1} q^\prime  }\right).    
\end{align*}

By writing $\beta= \alpha_1  q^\prime \overline{q^\prime} + \alpha_2  p^{r+r_1} \overline{ p^{r+r_1}}$ with $\alpha_{1}  \rm \mod \,\, p^{r+r_{1}}$ and $\alpha_{2} \, \rm \mod \,\, q^{\prime}$ in the above sum we see that the above character sum changes to
\begin{align*}
\sum_{\alpha_1 (p^{r+ r_1})}  \chi(\alpha_1)  e\left( \frac{  - (a + bq ) \alpha_1 \overline{q^\prime}}{ p^{\ell+r_1}  }  + \frac{ m \alpha_1 \overline{q^\prime} }{ p^{r+r_1}   }\right) \sum_{\alpha_2( q^\prime)} e  \left( \frac{  - (a + bq ) \alpha_2  \overline{p^{r+r_1}} p^{r- \ell}}{  q^\prime }  + \frac{ m \alpha_2  \overline{p^{r+r_1}}}{  q^\prime  }\right). 
\end{align*} 
Again, by writing $\alpha_1 = \beta_1 p^r + \beta_2$, where $\beta_2$ is modulo $ p^r$ and $ \beta_1$  modulo $ p^{r_1}$, the above sum becomes
\begin{align*}
 & \sum_{\beta_2(p^r)} \chi(\beta_2)  e\left( \frac{  - (a + bq ) p^{r-\ell}\beta_2 \overline{q^\prime }}{ p^{r+ r_1}  }  + \frac{ m \beta_2 \overline{q^\prime } }{ p^{r+ r_1}   }\right)  \sum_{\beta_1(p^{r_1})} e\left( \frac{  - (a + bq ) \beta_1 \overline{q^\prime} p^{r - \ell}}{ p^{r_1}  }  + \frac{ m \beta_1 \overline{q^\prime} }{ p^{r_1}   }\right) \\
& \hspace{2cm}  \times \sum_{\alpha_2( q^\prime)} e  \left( \frac{  - (a + bq ) \alpha_2  \overline{p^{r+r_1}} p^{r- \ell}}{  q^\prime }  + \frac{ m \alpha_2  \overline{p^{r+r_1}}}{  q^\prime  }\right).
\end{align*}
We execute sums over $\beta_1$ and $\alpha_{2}$ to transfer the above sum to 
$$q \, \mathbb{I}_{( m-  a p^{r- \ell} \equiv 0  (\textrm{mod}\ p^{r_1}  ))} \, \mathbb{I}_{( m-  a p^{r- \ell} \equiv 0  (\textrm{mod}\ q^{\prime}  ))} \sum_{\beta_2(p^r)} \chi(\beta_2)  e\left( \frac{\left( m - (a + bq ) p^{r-\ell} \right)\beta_2 \overline{q^\prime }}{ p^{r+ r_1}  }  \right).$$


We have $N\leq p^{r}$. Therefore, we have the inequality $p^{r_{1}} \leq q \leq Q = \sqrt{N/p^{\ell}} \leq p^{(r-\ell)/2} < p^{r-\ell}$. Thus, $\min\{r_{1},r-\ell\} = r_{1}$. Therefore the congruence $m-  a p^{r- \ell} \equiv 0  (\textrm{mod}\ p^{r_1}  )$ is same as $p^{r_{1} } \mid m$. 

Note that, since $\chi$ is a primitive character modulo $p^{r}$, the sum over $\beta_{2}$ is Gauss sum which vanishes unless $$ \left( \frac{m-(a+bq)p^{r-\ell}}{p^{r_{1}}},p\right)=1 \iff (m/p^{r_{1}},p)=1,$$
as $r_{1} < r-\ell$. In this case we have 

\begin{align*}
\sum_{\beta_2(p^r)} \chi(\beta_2)  e\left( \frac{\left( m - (a + bq ) p^{r-\ell} \right)\beta_2 \overline{q^\prime }}{ p^{r+ r_1}  }  \right)  = \chi (q^\prime) \overline{\chi} \left( \frac{m- (a + bq ) p^{r- \ell}}{p^{r_1}}\right) \sum_{\beta_2(p^r)}  \chi(\beta_2)  e\left( \frac{\beta_2 }{ p^{r}  }  \right) ,  
\end{align*}
Note that the last sum over $\beta_2$ is the Gauss sum. Which completes the proof of the lemma.

\end{proof}

After the Poisson summation formula, the sum $ S_{f,\chi}(N)$ is given by

	\begin{align}  
	S_{f,\chi}(N) &= \frac{1}{Qp^{\ell}} \int_{\mathbb{R}} \, \sum_{r_{1}=0 }^{\lfloor \frac{ \log Q}{\log p} \rfloor}    \sum_{\substack{1 \leq q^{\prime} \leq Q/p^{r_{1}} \\ (q^{\prime},p)=1}} \frac{g(p^{r_{1}}q^{\prime},x)}{p^{r_{1}}q^{\prime}}   \,\sideset{}{^\star}{\sum}_{a_{1} \rm mod \, p^{r_{1}}}\,  \sum_{b \,  \mathrm{mod} \, p^{\ell}} \notag\\
	&\times  \left\lbrace   \frac{ \tau_{\chi}   \chi (q^{\prime})  N }{ p^r }   \sum_{\substack{m^{\prime} \ll M_{0}/p^{r_{1}} \\ (m^{\prime},p)=1}}    \overline{\chi} \left( m^{\prime}- (a + bq ) p^{r- \ell-r_{1}} \right) \mathcal{I} (x , p^{r_{1}}q^{\prime}, p^{r_{1}}m^{\prime})\right\rbrace    \notag\\  
	&  \times \left\lbrace \sum_{n=1}^\infty  \lambda_f(n)   e\left(  \frac{ (a+ bq) n}{p^{\ell} q} \right) e\left(  \frac{ x n}{  p^{\ell} qQ} \right)  W\left( \frac{n}{N} \right)\right\rbrace dx + O_A \left(N^{-A} \right),\notag  
	\end{align}   
for any real $A >0$,
where  $a \, \rm mod \, q$ is determined in terms of $a_{1} \, \rm mod p^{r_{1}} $ and $m^{\prime}$. Indeed we have $a \, \equiv \,  m^{\prime} p^{2r_{1}} \overline{p^{r-\ell+r_{1}}}+a_{1} q^{\prime} \bar{q^{\prime}} \, \rm mod \, p^{r_{1}}q^{\prime}$. And $q=p^{r_{1}}q^{\prime}$. 

We now split the above expression for $S_{f,\chi}(N)$ as follows
$$S_{f,\chi}(N) = S_{f, \chi}(N; r_{1}=0 \text{ contribution}) + S_{f, \chi}(N; r_{1} \geq 1 \text{ contribution}).$$
Since $r_{1} \geq 1$ implies $(a+bq,p^{\ell}q)=1$, and then we can apply Voronoi summation formula directly. But, if $r_{1}=0$, then $a+bq $ may not be coprime to $p^{\ell}$. Therefore we can not directly apply Voronoi summation. So we need to work with these two different situations. Note that these cases can be dealt in a similar fashion. 

From now on we only focus on estimation of $S_{f, \chi}(N; r_{1}=0 \text{ contribution})$. In a similar way we can estimate $ S_{f, \chi}(N; r_{1} \geq 1 \text{ contribution})$ and even get better estimates in this case. 

\begin{align}\label{afterpoisson}
S_{f,\chi}(N; r_{1}=0 \, \text{contri}) &= \frac{1}{Qp^{\ell}} \int_{\mathbb{R}} \,     \sum_{\substack{1 \leq q \leq Q \\ (q,p)=1}} \frac{g(q,x)}{q}   \,  \sum_{b \,  \mathrm{mod} \, p^{\ell}} \notag\\
&\times \left\lbrace   \frac{ \tau_{\chi}   \chi (q)  N }{ p^r }   \sum_{\substack{m \ll M_{0} \\ (m,p)=1}}    \overline{\chi} \left( m- (a + bq ) p^{r- \ell} \right) \mathcal{I} (x , q, m)\right\rbrace    \notag\\  
&  \times \left\lbrace \sum_{n=1}^\infty  \lambda_f(n)   e\left(  \frac{ (a+ bq) n}{p^{\ell} q} \right) e\left(  \frac{ x n}{  p^{\ell} qQ} \right)  W\left( \frac{n}{N} \right)\right\rbrace dx,
\end{align}
where $a \equiv m \, \overline{p^{r-\ell}} \, \rm mod \, q$. Note that $(m,q)=1$.

\subsection{Application of Voronoi summation formula}
Now we appeal to an application of Vornoi summation formula on the  $n$-sum in \eqref{afterpoisson}. Recall from \eqref{nsum} that  $\mathcal{S}_{f}(N;a,b,q,x)$ is same as this $n$-sum. The application of Voronoi summation formula leads to the following lemma.

\begin{lemma} \label{applicavoronoi}
Let $(a+bq,p^{\ell}) = p^{\ell_{1}}$ for some $0 \leq \ell_{1} \leq \ell$. Then we have 
$$\mathcal{S}_{f}(N;a,b,q,x) = \frac{2 \pi i^{k} N^{3/4}}{p^{(\ell-\ell_{1})/2}q^{1/2}} \sum _{\varepsilon \in \{\pm \}} \sum _{1 \leq n \ll N_{0}} \frac{\lambda_{f}(n)}{n^{1/4}} e\left(-\frac{\overline{\left((a+bq)/p^{\ell_{1}}\right)}n}{p^{\ell -\ell_{1}}q}\right) \, \mathcal{J}(\varepsilon, q,x,n),$$
where  
$$\mathcal{J}(\varepsilon, q,x,n)= \int W_{1,\varepsilon}(y) \,e\left(\frac{xNy}{p^{\ell} qQ}\right) e\left(\frac{\varepsilon 2 \sqrt{nNy}}{p^{\ell-\ell_{1}}q}\right) \, \rm dy,$$
and $N_{0}= p^{\ell -2 \ell_{1}} \, N^{\epsilon}$.

\end{lemma}

\begin{proof}
An application of Voronoi summation transfers the  $\mathcal{S}_{f}(N;a,b,q,x) $ into
$$ \frac{2 \pi i^{k} }{p^{\ell -\ell_{1}} q} \sum_{n =1}^{\infty} \lambda_{f}(n) \, e\left(-\frac{\overline{\left((a+bq)/p^{\ell_{1}}\right)}n}{p^{\ell -\ell_{1}}q}\right) \, \int_{\mathbb{R}} W(\frac{y}{N}) \, e\left(\frac{xy}{p^{\ell} qQ}\right) \, J_{k-1}\left(\frac{4\pi \sqrt{yn}}{p^{\ell-\ell_{1}}q}\right) \, \rm d y,$$
where $k$ is the weight of the holomorphic Hecke-eigenform $f$ and $J_{k-1}(x)$ is the Bessel function. We make a change of variables $y/N \mapsto z$ in the above integration and use the expression 
$$J_{k-1}(x) = \frac{1}{\sqrt{x}} \sum_{\varepsilon \in \{\pm \}} W_{k, \varepsilon}(x) \, e^{i\varepsilon x},$$
where $x^{j} W^{(j)}_{k, \varepsilon}(x)\ll_{k,j} 1 $ for $x \gg 1$, for Bessel function to get 
$$\frac{N^{3/4} p^{(\ell -\ell_{1})/2} q^{1/2}}{n^{1/4}} \sum _{\varepsilon \in \{\pm \}} \int W_{1,\varepsilon}(y) \,e\left(\frac{xNy}{p^{\ell} qQ}\right) e\left(\frac{\varepsilon 2 \sqrt{nNy}}{p^{\ell-\ell_{1}}q}\right) \, \rm dy,$$
where $W_{1,\varepsilon}(y)= W(y) W_{k,\varepsilon}\left( \frac{4\pi \sqrt{nNy}}{p^{\ell-\ell_{1}}q}\right)$ which satisfies $y^{j} W^{(j)}_{1,\epsilon}(y) \ll_{k,j} 1$. We see , repeated integration by parts, that the above integral in negligibly small unless
$$1 \leq n \ll N_{0}= p^{\ell -2 \ell_{1}} \, N^{\epsilon}.$$
Thus the lemma follows.

\end{proof}

By noting the following identities 
\begin{align*}
\frac{\left(\overline{(a+bq)/p^{\ell_{1}}}\right)}{p^{\ell - \ell_{1}} q} &= \frac{\bar{a} p^{\ell_{1}} p^{\ell-\ell_{1}} \overline{p^{\ell - \ell_{1}}}+\overline{\left((a+bq)/p^{\ell_{1}}\right)}q\bar{q}}{p^{\ell_{1}-\ell_{1}} q} \\
&=  \frac{\overline{\left((a+bq)/p^{\ell_{1}}\right)} \, \bar{q} \, }{p^{\ell -\ell_{1} } } +\frac{\bar{m} p^{r} \overline{p^{2(\ell -\ell_{1})}}}{q},
\end{align*}
in the second equality  we have used the fact that where $a \equiv m \, \overline{p^{r-\ell}} \, \rm mod \, q$, 
and by arranging all terms  we get the following proposition.
\begin{proposition}
	We have
	
\begin{align*}
S_{f,\chi}(N; r_{1}=0 \, \text{contri}) &= \frac{\tau_{\chi} N^{7/4} 2 \pi i^{k} }{Q p^{r+\frac{3\ell }{2}}} \sum_{\varepsilon \in \{ \pm \}}\, \sum_{\ell_{1} =0}^{\ell } p^{\ell_{1}/2} \, T(\varepsilon, \ell_{1},N) 
\end{align*}
where 
\begin{align*}
T(\varepsilon, \ell_{1},N) &= \sum _{1 \leq n \ll N_{0}} \frac{\lambda_{f}(n)}{n^{1/4}} \sum_{ \substack{1 \leq q \leq Q \\ (q,p)=1}} \frac{\chi(q)}{q^{3/2}}  \sum_{\substack{m \ll M_{0} \\ (m,p)=1}}  \sideset{}{^\dagger}{\sum }_{ \beta \, \rm mod \, p^{\ell -\ell_{1}}}  \\
& \times   \overline{\chi} \left( m- (a + bq ) p^{r- \ell} \right)  e\left(-\frac{\overline{\left((a+bq)/p^{\ell_{1}}\right)} \, \bar{q}n \, }{p^{\ell -\ell_{1} } } - \frac{\bar{m} p^{r} \overline{p^{2(\ell -\ell_{1})}}n}{q}\right) \, \mathfrak{I}(\varepsilon, q,n,m),
\end{align*}
where $b= -a \bar{q}+\beta p^{\ell_{1}}$, the symbol $\dagger$ on the $b$-sum means that $((a+bq)/p^{\ell_{1}},p)=1$, and 
$$\mathfrak{I}(...)= \int_{\mathbb{R}} g(q,x) \, \mathcal{J}(\varepsilon, q,x,n) \, \mathcal{I}( q,x,m) \, \rm dx.$$
	
\end{proposition}

\subsection{Bounds for integral $\mathfrak{I}(\varepsilon, q,n,m)$}
In this subsection we give bounds for the integral $\mathfrak{I}(\varepsilon, q,n,m)$ which is useful when we will deal with small $q$ (note that phase functions appear in $\mathfrak{I}(\varepsilon, q,n,m)$ oscillates when $q$ is small).

\begin{lemma} \label{integralbound}
We have
$$\mathfrak{I}(\varepsilon, q,n,m) \ll \frac{p^{\ell} qQ}{N} \, N^{\epsilon}.$$
	
\end{lemma}
\begin{proof}
Recall that $\mathfrak{I}(\varepsilon, q,n,m)$ is give by
$$\mathfrak{I}(...)= \int_{\mathbb{R}} g(q,x) \, \mathcal{J}(\varepsilon, q,x,n) \, \mathcal{I}( q,x,m) \, \rm dx.$$
The function $q(q,x)$ is negligible unless $|x| \leq N^{\epsilon}$. Therefore we have 
\begin{align} \label{iyintegral}
\mathfrak{I}(...)= \int_{|x| \leq N^{\epsilon}}  \, g(q,x) \,& \int_{\mathbb{R}} W_{1,\varepsilon}(y) \,e\left(\frac{xNy}{p^{\ell} qQ}\right) e\left(\frac{\varepsilon 2 \sqrt{nNy}}{p^{\ell-\ell_{1}}q}\right) \notag\\
&\int_{\mathbb{R}} V (z)  e\left(  \frac{ - N x z}{  p^{\ell}  qQ} \right) e\left(  \frac{ - N m z }{  p^{r} q} \right) dz \, \rm dy \, \rm dx +O\left(N^{-2022}\right).
\end{align}
Now we consider the  $z$ integral. By repeated integration by parts we see that $z$ integral is negligible unless 
$$ \Big \vert \frac{Nx}{p^{\ell} qQ} + \frac{Nm}{p^{r} q} \Big \vert  \ll N^{\epsilon} \iff \Big \vert  x+\frac{mQ}{p^{r-\ell} } \Big \vert \ll \frac{p^{\ell} q Q}{N} N^{\epsilon}.$$

We know, by properties of the function $g(q,x)$, see Section \ref{circlemethod}, that 
$$g(q,x)=1 +O\left(N^{-2022}\right),$$
if $q\leq Q^{1-\epsilon}$ or $|x| \leq N^{-\epsilon}$. 
Therefore we divide $x$-integral into two parts. Indeed, we write
$$\mathfrak{I}(\varepsilon, q,n,m)= \left(\int_{ \substack{ |x|\leq N^{-\epsilon} \\\Big \vert  x+\frac{mQ}{p^{r-\ell} } \Big \vert \ll \frac{p^{\ell} q Q}{N} N^{\epsilon} }} + \int_{\substack{ N^{- \epsilon} \leq |x|\leq N^{\epsilon} \\\Big \vert  x+\frac{mQ}{p^{r-\ell} } \Big \vert \ll \frac{p^{\ell} q Q}{N} N^{\epsilon} }}\right) g(q,x) \lbrace ... \rbrace \, dx \, dy \, dz +O\left(N^{-2022}\right).$$
In the first integral we can replace $g(q,x)$ by $1$ up to a negligible error term. We treat everything else trivially in the first $x$-integral to get this integral to be
$$\ll \frac{p^{\ell} q Q}{N} N^{\epsilon}.$$

In the second $x$-integral, we have the condition that $ N^{-\epsilon} \leq |x| \leq N^{\epsilon}$. In this case we consider $y$-integral in \eqref{iyintegral}. In this integral we make change of variable $y \to y^{2}$, then the resulting expression of this integral is given by 
$$\int_{\mathbb{R}} 2 \, y \, W_{1,\varepsilon}(y^2) \, \, e\left(f(y)\right)  \, \rm dy,$$
where the phase function 
$$f(y)= \frac{xNy^{2}}{p^{\ell} qQ} +\frac{\varepsilon 2 \sqrt{nN}y}{p^{\ell-\ell_{1}}q}.$$
The stationary point $y_{0}$ of $f(y)$ is given by $y_{0}=  \varepsilon \sqrt{n} Qp^{\ell_{1}}/x\sqrt{N}$, and  we have
$$f^{(\prime \prime)}(y_{0}) = \frac{2xN}{p^{\ell} qQ}.$$
Thus we have $$ \frac{1}{\sqrt{|f^{(\prime \prime)}(y_{0})|}} \ll \sqrt{\frac{p^{\ell} qQ}{N}} N^{\epsilon}.$$
Therefore, by stationary method this $y$-integral is at most
$$\ll \sqrt{\frac{p^{\ell} qQ}{N}} N^{\epsilon}.$$
Thus, the second $x$-integral is at most
\begin{align*}
&\ll \sqrt{\frac{p^{\ell} qQ}{N}} N^{\epsilon} \, \int_{\substack{ N^{- \epsilon} \leq |x|\leq N^{\epsilon} \\\Big \vert  x+\frac{mQ}{p^{r-\ell} } \Big \vert \ll \frac{p^{\ell} q Q}{N} N^{\epsilon} }} \vert q(q,x)\vert \, dx \\
&\ll \frac{p^{\ell} qQ}{N} N^{\epsilon} \, \int_{\mathbb{R}} |g(q,x)|^{2} \, dx \\
&\ll \frac{p^{\ell} qQ}{N} N^{\epsilon}.
\end{align*}
We have used the $L^{2}$-bound for the function $g(q,x)$ from the Section \ref{circlemethod}.
\end{proof}

\section{Cauchy and Poisson}

An application of the Cauchy's inequality on the $n$ sum in $T(\varepsilon, \ell_{1},N)$ along with Ramanujan bound for the Fourier coefficients $\lambda_{f}(n)$ gives that

\begin{equation} \label{cauchyappli}
T(\varepsilon, \ell_{1},N) \ll N_{0}^{1/4} \, \Theta^{1/2}
\end{equation}
where
\begin{align*}
\Theta  = \sum_{n} W_{2} \left(\frac{n}{N_{0}}\right)&\Big\vert  \sum_{ \substack{1 \leq q \leq Q \\ (q,p)=1}} \frac{\chi(q)}{q^{3/2}}  \sum_{\substack{m \ll M_{0} \\ (m,p)=1}}  \sideset{}{^\dagger}{\sum }_{ \beta \, \rm mod \, p^{\ell -\ell_{1}}}  \\ & \times \overline{\chi} \left( m- (a + bq ) p^{r- \ell} \right)  e\left(-\frac{\overline{\left((a+bq)/p^{\ell_{1}}\right)} \, \bar{q}n \, }{p^{\ell -\ell_{1} } } - \frac{\bar{m} p^{r} \overline{p^{2(\ell -\ell_{1})}}n}{q}\right) \, \mathfrak{I}(...)\Big \vert^{2},
\end{align*}
where $W_{2}$ is smooth bump function supported in $[1,2]$.
After opening the absolute square and interchanging  sums we arrive at

\begin{align*}
\Theta = \mathop{\sum \sum }_{\substack{1 \leq q_1, q_2 \leq Q \\ (q_1q_2,p)=1} } \frac{\chi\left(q_{1}q_{2}\right)}{q_{1}^{3/2} q_{2}^{3/2}} &\mathop{\sum \sum}_{\substack{m_{1},m_{2} \ll M_0 \\(m_1m_2,p)=1 }} \sideset{}{^\dagger}{\sum }_{ \beta_{1} \, \rm mod \, p^{\ell -\ell_{1}}} \sideset{}{^\dagger}{\sum }_{ \beta_{2} \, \rm mod \, p^{\ell -\ell_{1}}} \\
&\times \overline{\chi} \left( m_{1}- (a_{1} + b_{1}q_{1} ) p^{r- \ell} \right) \chi \left( m_{2}- (a_{2} + b_{2}q_{2} ) p^{r- \ell}\right) \, T(...),
\end{align*}
where
\begin{align*}
T(m_{1},m_{2},q_{1},q_{2},b_{1},b_{2}) &= \sum_{n \in \mathbb{Z}} W_{2} \left(\frac{n}{N_{0}}\right) e\left(\frac{\overline{\left((a_{1}+b_{1}q_{1})/p^{\ell_{1}}\right)} \,\bar{q}_{1} n- \overline{\left((a_{2}+b_{2}q_{2})/p^{\ell_{1}}\right)} \,\bar{q}_{2} n }{p^{\ell -\ell_{1} }} \right)\\
& \times \left(\frac{\bar{m}_{1} p^{r}\overline{p^{2 (\ell -\ell_{1})}} n}{q_{1}} -\frac{\bar{m}_{2} p^{r}\overline{p^{2 (\ell -\ell_{1})}} n}{q_{2}} \right) \mathfrak{I}(\varepsilon, q_{1},n,m_{1}) \overline{\mathfrak{I}(\varepsilon, q_{2},n,m_{2})}.
\end{align*}
First we split the sum over $n$ into congruence classes modulo $p^{\ell -\ell_{1}} q_{1}q_{2}$. That is, for any $\alpha $ modulo $p^{\ell -\ell_{1}} q_{1} q_{2}$ we write $n = \alpha+ k p^{\ell -\ell_{1}} q_{1}q_{2}$ with $k \in \mathbb{Z}$. Then by applying Poisson summation on $k$ variable we arrive at the expression  
\begin{align*}
\frac{N_{0}}{p^{\ell -\ell_{1}} q_{1}q_{2}} \sum_{n \in \mathbb{Z}} \, &\sum_{ \alpha \, \mathrm{mod} \, p^{\ell -\ell_{1}} q_{1}q_{2}} e\left(\frac{\left(\overline{\left((a_{1}+b_{1}q_{1})/p^{\ell_{1}}\right)} \bar{q}_{1} q_{1}q_{2} -\overline{\left((a_{2}+b_{2}q_{2})/p^{\ell_{1}}\right)} \bar{q}_{2} q_{1}q_{2}\right) \alpha}{p^{\ell-\ell_{1}} q_{1}q_{2}}\right)\\
& \times e\left(\frac{\left(  p^{r}\overline{p^{2(\ell -\ell_{1})}}p^{\ell -\ell_{1}} \bar{m}_{1}q_{2}- p^{r}\overline{p^{2 (\ell- \ell_{1})}}p^{\ell -\ell_{1}} \bar{m}_{2}q_{1} +n\right) \, \alpha}{p^{\ell - \ell_{1}} q_{1}q_{2}}\right) \mathfrak{I}_{1}(n,q_{i},m_{i},\varepsilon),
\end{align*}
where
$$\mathfrak{I}_{1}(n,q_{i},m_{i},\varepsilon)= \int W_{2}(y) \, \mathfrak{I}(\varepsilon, q_{1},N_{0}y,m_{1}) \overline{\mathfrak{I}(\varepsilon, q_{2},N_{0}y,m_{2})} \, e\left(-\frac{nN_{0}y}{p^{\ell -\ell_{1}} q_{1}q_{2}}\right) \rm dy,$$
for $T(m_{1},m_{2},q_{1},q_{2},b_{1},b_{2})$. Note, by  repeated integration by parts, that  $\mathfrak{I}_{1}(n,q_{i},m_{i},\varepsilon)$ is negligibly small unless $$|n| \leq \frac{q_{1}q_{2}p^{\ell -\ell_{1}}}{N_{0}} N^{\epsilon}.$$
Therefore after executing the sum over $\alpha$, the value of $T(...) $ is given by
$$ N_{0} \mathop{\sum_{ \substack{|n| \leq \frac{q_{1}q_{2}p^{\ell -\ell_{1}}}{N_{0}} N^{\epsilon} }} \, \mathfrak{I}_{1}(...),}_{\overline{\left((a_{1}+b_{1}q_{1})/p^{\ell_{1}}\right)} \bar{q}_{1} q_{1}q_{2} -\overline{\left((a_{2}+b_{2}q_{2})/p^{\ell_{1}}\right)} \bar{q}_{2} q_{1}q_{2} + p^{r}\overline{p^{2 (\ell -\ell_{1})}}p^{\ell -\ell_{1}} \bar{m}_{1}q_{2}- p^{r}\overline{p^{2 (\ell -\ell_{1})}}p^{\ell-\ell_{1}} \bar{m}_{2}q_{1} +n \, \equiv \, 0 \, \mathrm{mod} \, p^{\ell-\ell_{1}} q_{1}q_{2}}$$
upto a negligible error term.

The above congruence relation gives that
$$\overline{\left(\left(a_{1}+b_{1}q_{1}\right)/p^{\ell_{1}}\right)} \, q_{2} - \overline{\left(\left(a_{2}+b_{2}q_{2}\right)/p^{\ell_{1}}\right)} \, q_{1}+n \, \equiv \, 0 \, \mathrm{mod} \, p^{\ell -\ell_{1}},$$
and
$$ p^{r -\ell+\ell_{1}} \bar{m}_{1} q_{2}-p^{r -\ell+\ell_{1}} \bar{m}_{2} q_{1} +n \, \equiv \, 0 \, \mathrm{mod} \, q_{1}q_{2}.$$ 

After changing the variables $(a_{1}+b_{1}q_{1})/p^{\ell_{1}} \mapsto \alpha_{1}$ and $(a_{2}+b_{2}q_{2})/p^{\ell_{1}} \mapsto \alpha_{2}$, we see that $\Theta $ is given by

\begin{align} \label{finalthetavalue}
\Theta = N_{0} \mathop{\sum \sum }_{\substack{1 \leq q_1, q_2 \leq Q \\ (q_1q_2,p)=1} } \frac{\chi\left(q_{1}q_{2}\right)}{q_{1}^{3/2} q_{2}^{3/2}} &\mathop{\mathop{\sum \sum}_{\substack{m_{1},m_{2} \ll M_0 \\(m_1m_2,p)=1 }} \sideset{}{^\star}{\sum }_{ \alpha_{1} \, \rm mod \, p^{\ell -\ell_{1}}} \sideset{}{^\star}{\sum }_{ \alpha_{2} \, \rm mod \, p^{\ell -\ell_{1}}} \sum_{ \substack{|n| \leq \frac{q_{1}q_{2}p^{\ell -\ell_{1}}}{N_{0}} N^{\epsilon} }}}_{\substack{\bar{\alpha}_{1} q_{2} - \bar{\alpha}_{2} q_{1}  + n \, \equiv \, 0 \, \mathrm{mod} \, p^{\ell-\ell_{1}} \\  p^{r -\ell+\ell_{1}} \bar{m}_{1} q_{2}-p^{r -\ell+\ell_{1}} \bar{m}_{2} q_{1} +n \, \equiv \, 0 \, \mathrm{mod} \, q_{1}q_{2}  }} \notag \\
&\times \overline{\chi} \left( m_{1}- \alpha_{1} p^{r- \ell+\ell_{1}} \right) \chi \left( m_{2}- \alpha_{2} p^{r- \ell+\ell_{1}}\right) \mathfrak{I}_{1}(n,q_{i},m_{i},\varepsilon).
\end{align}


\subsection{Zero frequency $n=0$: }
We write $\Theta_{\text{zero}}$ for the contribution of the zero frequency to $\Theta$. In the following lemma we give estimates for $\Theta_{\text{zero}}$.

\begin{lemma} \label{zerofrequency}
We have 
$$\Theta_{\text{zero}} \ll \frac{p^{r+\frac{5\ell}{2} -3\ell_{1}}}{N^{3/2}} N^{\epsilon},$$
provided that $ N \geq p^{r -(\ell -\ell_{1})}$.
\end{lemma}

\begin{proof}
For $n=0$, the congruence conditions in \eqref{finalthetavalue} becomes 
\begin{align*}
&\bar{\alpha}_{1} q_{2} - \bar{\alpha}_{2} q_{1}  \, \equiv \, 0 \, \mathrm{mod} \, p^{\ell-\ell_{1}}, \text{and} \\
& \bar{m}_{1} q_{2}-\bar{m}_{2} q_{1}  \, \equiv \, 0 \, \mathrm{mod} \, q_{1}q_{2}.
\end{align*}
From the second congruence we infer that $q_{1} \mid q_{2}$ and $q_{2} \mid q_{1}$ which implies that $q_{1}=q_{2}=q$, and we also have that $q \mid m_{1}-m_{2}$. Then from the first congruence we immediately see that $\alpha_{1} \, \equiv  \,  \alpha_{2} \, \rm mod \, p^{\ell -\ell_{1}}$. Therefore $\Theta_{\text{zero}}$ is given by
\begin{align*}
N_{0} \sum_{ \substack{1 \leq q \leq Q \\(q,p)=1 }} \frac{\chi(q^{2})}{q^{3}} &\mathop{\mathop{\sum \sum}_{\substack{m_{1},m_{2} \ll M_0 \\ q\mid m_{1}-m_{2} \\ (m_1m_2,p)=1 }}} \chi(\bar{m}_{1} m_{2}) \, \sideset{}{^\star}{\sum }_{ \alpha_{1} \, \rm mod \, p^{\ell -\ell_{1}}} \\
&\times \overline{\chi} \left( 1- \alpha_{1} \bar{m}_{1} p^{r- (\ell-\ell_{1})} \right) \chi \left( 1- \alpha_{1} \bar{m}_{2} p^{r- (\ell-\ell_{1})}\right) \mathfrak{I}_{1}(n,q,m_{i},\varepsilon).
\end{align*}

For $m_{1} \neq m_{2}$, we evaluate the character sum over $\alpha_{1}$. To this end note that we have $\chi(1+zp^{r-(\ell- \ell_{1})}) = e\left(\frac{-A_{1} p^{2 r-2(\ell - \ell_{1})} \, z^{2}-A_{2} p^{r-(\ell-\ell_{1})} \, z}{p^{r}}\right)$, for some integers $A_{1}$ and $A_{2}$ which are coprime to $p$, as our choice of $\ell$ satisfies the condition $r-\ell \geq r/3$ which is same as $\ell \leq 2r/3$  (see \cite[Lemma 13]{mil}). Thus, the  $\alpha_{1}$ sum is given by
\begin{equation} \label{alphasum}
\sideset{}{^\star}{\sum }_{ \alpha \, \rm mod \, p^{\ell -\ell_{1}}} e\left(\frac{Y_{1}\alpha^{2}+Y_{2}\alpha}{p^{r}}\right),
\end{equation}
where
\begin{align*}
&Y_{1} =  A_{1} p^{2 r-2 (\ell - \ell_{1})}  \left(\bar{m}^{2}_{2}- \bar{m}^{2}_{1}\right),  \text{and}\\
&Y_{2}= A_{2} p^{r-(\ell - \ell_{1})} \left( \bar{m}_{1}-\bar{m}_{2}\right).
\end{align*}

Note that this character sum is same as
$$\sum_{\alpha \rm \, mod \, p^{\ell -\ell_{1}}} \,e\left(\frac{A_{1} p^{r-(\ell -\ell_{1})} (\bar{m}_{1}^{2}-\bar{m}_{2}^{2}) \alpha^{2}+ A_{2} (\bar{m}_{1}-\bar{m}_{2})\alpha}{p^{\ell -\ell_{1}}}\right) +O\left(\ell\right) .$$
This sum is quadratic Gauss sum which vanishes unless
$$ p^{\ell - \ell_{1}} \mid (m_{1} -m_{2}).$$
Thus the value of above quadratic Gauss is 
$$p^{\ell - \ell_{1}} \, \mathbb{I}_{m_{1} \, \equiv \, m_{2} \, \rm mod \, p^{\ell-\ell_{1}}}.$$


By substituting the bound for the character sum over $\alpha$  and the bound for the integral $$\mathfrak{I}_{1}(n,q,m_{i},\varepsilon) \ll \frac{p^{2 \ell} q^{2} Q^{2}}{N^{2}} N^{\epsilon}$$
in $\Theta_{\text{zero}}$, we see that
\begin{align*}
\Theta_{\text {zero}} & \ll N_{0} \Big \lbrace \, p^{\ell - \ell_{1}}\sum_{ \substack{1 \leq q \leq Q \\(q,p)=1 }} \frac{1}{q^{3}} \mathop{\mathop{\sum \sum}_{\substack{m_{1},m_{2} \ll M_0 \\ q\mid m_{1}-m_{2} \\ (m_1m_2,p)=1 }}}  \frac{p^{2 \ell} q^{2} Q^{2}}{N^{2}} \mathbb{I}_{m_{1} \, \equiv \, m_{2} \, \rm mod \, p^{\ell-\ell_{1}}} 
\\
& +  \, \sum_{ \substack{1 \leq q \leq Q \\(q,p)=1 }} \frac{1}{q^{3}} \mathop{\mathop{\sum \sum}_{\substack{m_{1},m_{2} \ll M_0 \\ q\mid m_{1}-m_{2} \\ (m_1m_2,p)=1 }}}  \frac{p^{2 \ell} q^{2} Q^{2}}{N^{2}} \Big \rbrace \, N^{\epsilon} \\
& \ll \frac{p^{r+\frac{5\ell}{2} -3\ell_{1}}}{N^{3/2}} N^{\epsilon} +  \frac{p^{2r+\frac{3\ell}{2} -2\ell_{1}}}{N^{5/2} } N^{\epsilon} \\
&\ll \frac{p^{r+\frac{5\ell}{2} -3\ell_{1}}}{N^{3/2}} N^{\epsilon},
\end{align*}
where in the second inequality the first term corresponds to the contribution of $m_{1} = m_{2}$, and second one corresponds to the contribution of $m_{1} \neq m_{2}$, and in the last inequality we have used the assumption that $N \geq p^{r -(\ell -\ell_{1})}$. This concludes the proof the lemma.


\end{proof}
Let $T_{0}(\varepsilon, \ell_{1},N)$ and $S_{f,\chi}(N; r_{1}=0 \, \text{contri}, \Theta_{\text {zero}}) $ denote the contribution of $\Theta_{\text{zero}}$ to $T(\varepsilon, \ell_{1},N)$ and $S_{f,\chi}(N; r_{1}=0 \, \text{contri}) $ respectively. Then we have that 
$$T_{0}(\varepsilon, \ell_{1},N) \ll \frac{p^{\frac{r+3 \ell -4 \ell_{1}}{2}} }{N^{3/4}}  N^{\epsilon}$$
and consequently we have that
\begin{align*}
S_{f,\chi}(N; r_{1}=0 \, \text{contri}, \Theta_{\text {zero}}) \ll \sqrt{N} p^{\ell/2} \, N^{\epsilon},
\end{align*}
provided $N \geq p^{r- \ell }$. We record this as the following proposition.

\begin{proposition} \label{zerocontributiion}
We have 
$$S_{f,\chi}(N; r_{1}=0 \, \text{contri}, \Theta_{\text {zero}}) \ll \sqrt{N} p^{\ell/2} \, N^{\epsilon},$$
\end{proposition}
provided $N \geq p^{r- \ell }$.

\subsection{Non-zero frequency $n \neq 0$:}
Assume that $n \neq 0$. In this case we determined $\alpha_{2} \, \mathrm{mod} \, p^{\ell}$ and write $m_{1}, m_{2}$ in terms of $q_{1},q_{2}$ and $n$ modulo $q_{1},q_{2}$ respectively using congruences. Indeed, 
$$\alpha_{2} \, \equiv \,  q_{1} \overline{ \left(\bar{\alpha}_{1} q_{2} +n\right)} \, \mathrm{mod} \, p^{\ell-\ell_{1}},$$
and 
$$m_{1} \, \equiv \, - \bar{n} p^{r -(\ell -\ell_{1})} q_{2} \, \mathrm{mod} \, q_{1}, \quad m_{2} \, \equiv \, \bar{n} p^{r-(\ell -\ell_{1})} q_{1}\, \mathrm{mod} \,q_{2}.$$
By writing $m_{1} = - \bar{n} p^{r -(\ell -\ell_{1})} q_{2} +r_{1}q_{1}$ and $m_{2} = \bar{n} p^{r-(\ell -\ell_{1})} q_{1} +r_{2}q_{2}$,  we see that

\begin{align*}
\Theta_{\text{non-zero}} = N_{0} & \,\mathop{\sum \sum }_{\substack{1 \leq q_1, q_2 \leq Q \\ (q_1q_2,p)=1} } \frac{\chi\left(q_{1}q_{2}\right)}{q_{1}^{3/2} q_{2}^{3/2}}\, \sum_{ \substack{0< |n| \leq \frac{q_{1}q_{2}p^{\ell -\ell_{1}}}{N_{0}} N^{\epsilon} }}  \,  \mathop{\sum \sum}_{\substack{|r_{1}| \leq \frac{p^{r}}{N}N^{\epsilon} \\ |r_{2}| \leq \frac{p^{r}}{N}N^{\epsilon}}} \mathcal{C}(r_{1},r_{2},q_{1},q_{2},n) \,  \mathfrak{I}_{1}(n,q_{i},m_{i},\varepsilon),
\end{align*}
where $\mathcal{C}(...)$ is given by
$$ \sideset{}{^\star}{\sum}_{\alpha \, \mathrm{mod} \, p^{(\ell -\ell_{1})}} \overline{\chi} \left( - \bar{n} p^{r -(\ell -\ell_{1})} q_{2} +r_{1}q_{1} - \alpha p^{r- (\ell -\ell_{1})} \right) \, \chi \left( \bar{n} p^{r-(\ell -\ell_{1})} q_{1} +r_{2}q_{2}-q_{1} \overline{ \left(\bar{\alpha} q_{2} +n\right)}  p^{r- (\ell -\ell_{1})}\right).$$

\subsection{Evaluation of sum over $\alpha$}
The $\alpha$ sum  is given by
$$\mathcal{C}(...) = \sideset{}{^\star}{\sum}_{ \alpha \, \mathrm{md} \, p^{(\ell - \ell_{1})}} \, \overline{\chi} \left( r_{1}q_{1} + \left(- \alpha - \bar{n}  q_{2}\right)p^{r- (\ell - \ell_{1})} \right) \chi \left( r_{2}q_{2}+ \left(-q_{1} \overline{ \left(\bar{\alpha} q_{2} +n\right)} + \bar{n} q_{1} \right) p^{r- (\ell - \ell_{1})}\right). $$
Note that $\chi(1+zp^{r-(\ell - \ell_{1})}) = e\left(\frac{-A_{1} p^{2 r-2(\ell - \ell_{1})} \, z^{2}-A_{2} p^{r-(\ell - \ell_{1})} \, z}{p^{r}}\right)$ for some integers $A_{i}$'s which are coprime to $p$, as our choice of $\ell$ satisfies the condition $r-\ell \geq r/3$ which is same as $\ell \leq 2r/3$. Thus, the character sum $\mathcal{C}(...)$ is same as
\begin{align*}
& \overline{\chi}(r_{1}q_{1}) \, \chi(r_{2}q_{2}) \, \,  e\left(\frac{A_{1}\overline{n} \left(\overline{r_{1}q_{1}} \, q_{2}+ \overline{r_{2}q_{2}} \, q_{1}\right)+\left(-A_{2} \bar{n}^{2} \bar{q}_{1}^{2} \bar{r}_{1}^{2} q_{2}^2 +A_{2} \bar{n}^{2} \bar{r}_{2}^{2} \bar{q}_{2}^{2} q_{1}^{2} \right)p^{r-(\ell - \ell_{1})}}{p^{\ell - \ell_{1}}}\right) \\
& \sideset{}{^\star}{\sum}_{\substack{ \alpha \, \mathrm{mod} \, p^{\ell - \ell_{1}} \\ \bar{\alpha} q_{2}+n \, \not \equiv \, 0 \, \mathrm{mod} \, p^{\ell - \ell_{1}}}} e\left(\frac{\left(A_{1} \overline{r_{1}q_{1}} -2A_{2} \bar{n} \bar{r}_{1}^{2} \bar{q}_{1}^{2} q_{2} p^{r -(\ell - \ell_{1})}\right) \alpha - \left(A_{1} \overline{r_{2}q_{2}} q_{1}+2A_{2} \bar{n} \bar{r}_{2}^{2} \bar{q}_{2}^{2} q_{1}^{2} p^{r-(\ell - \ell_{1})}\right)\overline{ \left(\bar{\alpha} q_{2} +n\right)} }{p^{\ell - \ell_{1}}}\right) \\
&  \hspace{2.5cm} \times   e\left(\frac{- A_{2} \bar{r}_{1}^{2} \bar{q}_{1}^{2} p^{r-(\ell - \ell_{1})} \alpha^{2} + A_{2} \bar{r}_{2}^{2}\bar{q}_{2}^{2} q_{1}^{2} p^{r-(\ell - \ell_{1})} \overline{ \left(\bar{\alpha} q_{2} +n\right)}^{2}}{p^{\ell - \ell_{1}}}\right).
\end{align*}
The above $\alpha$ sum is equals to
\begin{align*}
\sideset{}{^\star}{\sum}_{\substack{ \alpha \, \mathrm{mod} \, p^{\ell - \ell_{1}} \\ \alpha+ 1 \, \not \equiv \, 0 \, \mathrm{mod} \, p^{\ell - \ell_{1}}}} \,& e\left(\frac{\left(A_{1} \bar{n} \overline{r_{1}q_{1}} -2A_{2} \bar{n}^{2} \bar{r}_{1}^{2} \bar{q}_{1}^{2} q_{2} p^{r -(\ell - \ell_{1})}\right) \bar{\alpha} - \left(A_{1} \bar{n}\overline{r_{2}q_{2}} q_{1}+2A_{2} \bar{n}^{2} \bar{r}_{2}^{2} \bar{q}_{2}^{2} q_{1}^{2} p^{r-(\ell - \ell_{1})}\right)\overline{ \left(\alpha+1\right)} }{p^{\ell - \ell_{1}}}\right) \\
& \times   e\left(\frac{- A_{2} \bar{n}^{2}\bar{r}_{1}^{2} \bar{q}_{1}^{2} p^{r-(\ell - \ell_{1})} \bar{\alpha}^{2} + A_{2} \bar{n}^{2} \bar{r}_{2}^{2}\bar{q}_{2}^{2} q_{1}^{2} p^{r-(\ell - \ell_{1})} \overline{ \left(\alpha+1\right)}^{2}}{p^{\ell - \ell_{1}}}\right).
\end{align*}
We assume that $(\ell - \ell_{1})$ is an even positive integer to make exposition simpler and to keep ideas clear. We now evaluate the above sum by splitting the $\alpha$ variable.  Indeed, we write 
$$\alpha = \alpha_{1}+ \, \alpha_{2} \, p^{(\ell - \ell_{1}) /2}, \, \, \text{with} \, \, \quad \alpha_{1} \left(\neq 0, \, \neq -1\right) \, \mathrm{mod} \, p^{(\ell - \ell_{1}) /2}, \, \, \, \, \, \alpha_{2} \, \mathrm{mod} \, p^{(\ell - \ell_{1}) /2}.$$
Thus the  $\alpha$ sum can be written as
\begin{align*}
&\sideset{}{^\star}{\sum}_{\substack{\alpha_{1} \, \mathrm{mod} \, p^{(\ell - \ell_{1}) /2} \\ \alpha_{1} +1 \, \not \equiv \, 0 \, \mathrm{mod} \,  p^{(\ell - \ell_{1}) /2}}} e\left(\frac{X_{1} \, \bar{\alpha_{1}} + X_{2} \, \overline{\left(\alpha_{1}+1\right)} + X_{3} \, \bar{\alpha_{1}}^{2}+  X_{4} \, \overline{\left(\alpha_{1}+1\right)}^{2}}{p^{\ell -\ell_{1}}}\right) \\
& \hspace{.5cm} \times \sum_{\alpha_{2} \, \mathrm{mod} \, p^{(\ell -\ell_{1}) /2}} e\left(-\frac{\left(X_{1} \, \bar{\alpha}^{2}_{1} + X_{2} \, \overline{\left(\alpha_{1}+1\right)}^{2} \right) \alpha_{2}}{p^{(\ell - \ell_{1}) /2}}\right) \\
&= p^{(\ell -\ell_{1}) /2} \, \sideset{}{^\star}{\sum}_{\substack{\alpha_{1} \, \mathrm{mod} \, p^{(\ell -\ell_{1}) /2} \\ \alpha_{1} +1 \, \not \equiv \, 0 \, \mathrm{mod} \,  p^{(\ell -\ell_{1}) /2} \\ X_{1} \, \bar{\alpha}^{2}_{1} + X_{2} \, \overline{\left(\alpha_{1}+1\right)}^{2} \, \equiv 0 \, \mathrm{mod} \, p^{(\ell -\ell_{1})/2}}} \, e\left(\frac{X_{1} \, \bar{\alpha_{1}} + X_{2} \, \overline{\left(\alpha_{1}+1\right)} + X_{3} \, \bar{\alpha_{1}}^{2}+  X_{4} \, \overline{\left(\alpha_{1}+1\right)}^{2}}{p^{\ell -\ell_{1}}}\right),
\end{align*}
where
$$X_{1}= A_{1} \bar{n} \overline{r_{1}q_{1}} -2A_{2} \bar{n}^{2} \bar{r}_{1}^{2} \bar{q}_{1}^{2} q_{2} p^{r -(\ell -\ell_{1})},$$
$$X_{2}= - \left(A_{1} \bar{n}\overline{r_{2}q_{2}} q_{1}+2A_{2} \bar{n}^{2} \bar{r}_{2}^{2} \bar{q}_{2}^{2} q_{1}^{2} p^{r-(\ell -\ell_{1})}\right),$$
$$X_{3}= - A_{2} \bar{r}_{1}^{2} \bar{q}_{1}^{2} p^{r-(\ell -\ell_{1})},$$
and
$$X_{4} = A_{2} \bar{n}^{2} \bar{r}_{2}^{2}\bar{q}_{2}^{2} q_{1}^{2} p^{r-(\ell -\ell_{1})}.$$
Note that $X_{1} \, \equiv \, A_{1} \bar{n} \overline{r_{1}q_{1}} \, \mathrm{mod} \,\,  p^{(\ell -\ell_{1})/2}$, and $X_{2} \, \equiv \, - A_{1} \bar{n}\overline{r_{2}q_{2}} q_{1} \, \mathrm{mod} \, p^{(\ell -\ell_{1})/2}$ as $r-(\ell -\ell_{1})\geq (\ell -\ell_{1})/2$.

Thus this $\alpha$ sum is given by

\begin{align*}
p^{(\ell -\ell_{1}) /2} \, & e\left(\frac{\left( A_{1} \bar{n} \overline{r_{1}q_{1}} -2A_{2} \bar{n}^{2} \bar{r}_{1}^{2} \bar{q}_{1}^{2} q_{2} p^{r -(\ell-\ell_{1})}\right) \left(\left(r_{2} \bar{r}_{1}\right)^{1/2}-1\right)}{p^{\ell -\ell_{1}}}\right) \\ & \times e\left(-\frac{ \left( A_{1} \bar{n}\overline{r_{2}q_{2}} q_{1}+2A_{2} \bar{n}^{2} \bar{r}_{2}^{2} \bar{q}_{2}^{2} q_{1}^{2} p^{r-(\ell-\ell_{1})}\right) \left(1- \overline{\left(r_{2}\bar{r}_{1}\right)^{1/2}}\right)}{p^{\ell-\ell_{1}}}\right)\\
&\times  e\left(\frac{- A_{2} \bar{r}_{1}^{2} \bar{q}_{1}^{2} p^{r-(\ell-\ell_{1})} \left(\left(r_{2} \bar{r}_{1}\right)^{1/2}-1\right)^{2} - A_{2} \bar{r}_{1}^{2} \bar{q}_{1}^{2} p^{r-(\ell-\ell_{1})} \left(1- \overline{\left(r_{2}\bar{r}_{1}\right)^{1/2}}\right)^{2} }{p^{\ell-\ell_{1}}}\right),
\end{align*}
if $r_{2}\bar{r}_{1} \, \equiv \, \square \, \mathrm{mod} \, p^{(\ell-\ell_{1})/2}$, other wise this $\alpha $ sum is zero. Note that $r_{2}\bar{r}_{1}$ is square modulo $p^{(\ell-\ell_{1})/2}$ if and only if $r_{2}\bar{r}_{1}$ is square modulo $p$. Any $m $ modulo $p$ be such that $r_{2}\bar{r}_{1} \, \equiv \, m^{2} \, \mathrm{mod} \, p$ can be uniquely extended to modulo $p^{(\ell-\ell_{1})/2}$ with the property that $r_{2}\bar{r}_{1} \, \equiv \, m^{2} \, \mathrm{mod} \, p^{(\ell-\ell_{1})/2}$.

Therefore, we have
\begin{align*}
\mathcal{C}(...) &=p^{(\ell-\ell_{1}) /2} \, \,  \mathbb{I}_{r_{2}\bar{r}_{1} \, \equiv \, \square \, \mathrm{mod} \, p} \, \, \overline{\chi}(r_{1}q_{1}) \, \chi(r_{2}q_{2}) \, \,  e\left(\frac{A_{1}\overline{n} \left(\overline{r_{1}q_{1}} \, q_{2}+ \overline{r_{2}q_{2}} \, q_{1}\right)}{p^{\ell - \ell_{1}}}\right) \\
&\times e\left(\frac{\left(-A_{2} \bar{n}^{2} \bar{q}_{1}^{2} \bar{r}_{1}^{2} q_{2}^2 +A_{2} \bar{n}^{2} \bar{r}_{2}^{2} \bar{q}_{2}^{2} q_{1}^{2} \right)p^{r-(\ell - \ell_{1})}}{p^{\ell -\ell_{1}}}\right) \\
& \times e\left( \frac{\left( A_{1} \bar{n} \overline{r_{1}q_{1}} -2A_{2} \bar{n}^{2} \bar{r}_{1}^{2} \bar{q}_{1}^{2} q_{2} p^{r -(\ell-\ell_{1})}\right) \left(\left(r_{2} \bar{r}_{1}\right)^{1/2}-1\right)}{p^{\ell -\ell_{1}}}\right)\\
&\times e\left(-\frac{ \left( A_{1} \bar{n}\overline{r_{2}q_{2}} q_{1}+2A_{2} \bar{n}^{2} \bar{r}_{2}^{2} \bar{q}_{2}^{2} q_{1}^{2} p^{r-(\ell-\ell_{1})}\right) \left(1- \overline{\left(r_{2}\bar{r}_{1}\right)^{1/2}}\right)}{p^{\ell-\ell_{1}}}\right)\\
&\times  e\left(\frac{- A_{2} \bar{r}_{1}^{2} \bar{q}_{1}^{2} p^{r-(\ell-\ell_{1})} \left(\left(r_{2} \bar{r}_{1}\right)^{1/2}-1\right)^{2} - A_{2} \bar{r}_{1}^{2} \bar{q}_{1}^{2} p^{r-(\ell-\ell_{1})} \left(1- \overline{\left(r_{2}\bar{r}_{1}\right)^{1/2}}\right)^{2} }{p^{\ell-\ell_{1}}}\right).
\end{align*}

\subsection{The sum over $r_{2}$}
We now consider the $r_{2}$ sum which is given by
$$\Delta(n, q_{i}, r_{1},N, \varepsilon) = \sum_{\substack{ |r_{2}| \leq p^{r}/N \\ r_{2} \bar{r}_{1} \, \equiv \, \square \, \mathrm{mod} \, p}} \,  \chi(r_{2}) \, e\left(\frac{g(r_{2})}{p^{\ell -\ell_{1}}}\right) \, \mathfrak{I}_{1}(n,q_{i},r_{1},r_{2},\varepsilon),$$
where
\begin{align*}
g(r_{2})= &A_{1} \bar{n}  \bar{q_{2}} q_{1} \bar{r}_{2} + A_{2} \bar{n}^{2}  \bar{q}_{2}^{2} q_{1}^{2} p^{r-(\ell-\ell_{1})} \bar{r}_{2}^{2}+\left( A_{1} \bar{n} \overline{r_{1}q_{1}} -2A_{2} \bar{n}^{2} \bar{r}_{1}^{2} \bar{q}_{1}^{2} q_{2} p^{r -(\ell-\ell_{1})}\right) \left(r_{2} \bar{r}_{1}\right)^{1/2} \\
&-\left( A_{1} \bar{n}\overline{r_{2}q_{2}} q_{1}+2A_{2} \bar{n}^{2} \bar{r}_{2}^{2} \bar{q}_{2}^{2} q_{1}^{2} p^{r-(\ell-\ell_{1})}\right) \left(1- \overline{\left(r_{2}\bar{r}_{1}\right)^{1/2}}\right)- A_{2} \bar{r}_{1}^{2} \bar{q}_{1}^{2} p^{r-(\ell-\ell_{1})} \left(r_{2} \bar{r}_{1}-2 (r_{2} \bar{r}_{1})^{1/2}\right) \\
& - A_{2} \bar{r}_{1}^{2} \bar{q}_{1}^{2} p^{r-(\ell-\ell_{1})} \left( r_{1} \bar{r}_{2}-2 (r_{2}\bar{r}_{1})^{1/2} \right).
\end{align*}

By taking dyadic sub division we see that this sum is at most
$$\Delta(n, q_{i}, r_{1},N, \varepsilon) \ll N^{\epsilon} \sup_{R \leq \frac{p^{r}}{N}} |T(R)|,$$
where 
$$T(R) = \sum_{\substack{  R \leq r_{2} \leq 2R \\ r_{2} \bar{r}_{1} \, \equiv \, \square \, \mathrm{mod} \, p}} \,  \chi(r_{2}) \, e\left(\frac{g(r_{2})}{p^{\ell-\ell_{1}}}\right) \, \mathfrak{I}_{1}(n,q_{i},r_{1},r_{2},\varepsilon).$$

%
%
\begin{remark}
Note that we have 
$$\frac{\partial}{\partial r_{2}} \mathfrak{I}_{1}(n,q_{i},r_{1},r_{2},\varepsilon) \ll  \frac{N}{p^{r}} \, \frac{p^{2 \ell} q^{2} Q^{2}}{N^{2}} N^{\epsilon},$$
so we can ignore the integral $\mathfrak{I}_{1}(...)$ while estimating $T(R)$. 
\end{remark}

In the following lemma we give estimate for $T(R)$.
\begin{lemma}
We have
$$T(R) \ll p^{14/15} \, p^{r/30} \, R^{1/5} N^{\epsilon},$$
for any $\epsilon>0$.
\end{lemma}

\begin{proof}
Let $\kappa$ be a large positive integer but fixed. Then we have that
\begin{align*}
T(R) &= \frac{1}{2} \sum_{1 \leq m \leq p^{\kappa}} \, \sum_{ \substack{ R \leq r_{2} \leq 2R \\ r_{2} \, \equiv \, r_{1} \, m^{2} \, \mathrm{mod} \, p^{\kappa}}} \chi\left( r_{2} \right) \, e\left(\frac{g(r_{2})}{p^{\ell-\ell_{1}}}\right) \\
&= \frac{1}{2} \sum_{1 \leq m \leq p^{\kappa}} \chi(r_{1}m^{2}) \, \sum_{\frac{R-r_{1}m^{2}}{p^{\kappa}} \leq t \leq \frac{2R-r_{1}m^{2}}{p^{\kappa}}} \, \chi \left(1+\bar{r}_{1} \bar{m}^{2} p t\right) \, e\left(\frac{g(r_{1}m^{2}+tp)}{p^{\ell-\ell_{1}}}\right) \\
&=\frac{1}{2} \sum_{1 \leq m \leq p^{\kappa}} \chi(r_{1}m^{2}) \ \sum_{\frac{R-r_{1}m^{2}}{p^{\kappa}} \leq t \leq \frac{2R-r_{1}m^{2}}{p^{\kappa}}} e\left(\frac{f(t)}{p^{r}}\right),
\end{align*}
where $f(t) = a_{0} \log_{p} \left(1+p^{\kappa}r_{1}\bar{m}^{2} t \right) + p^{r-(\ell-\ell_{1})} g(r_{1}m^{2}+tp)$. Note that 
$$f^{\prime}(t) = p^{\kappa} \, a_{0}r_{1}\bar{m}^{2} \left(1+p^{\kappa}r_{1}\bar{m}^{2} t\right)^{-1}+ p^{r-(\ell-\ell_{1})} \, h(t),$$
where $h(t)=p \, g^{\prime}(r_{1}m^{2}+pt)$. Our phase function $f$ is in the class ${\bf F}(\kappa, 1,\kappa, \lambda, u)$ for arbitrarily large positive $\lambda$ and positive integer $u$ but fixed (See \cite[Section 3]{mil}, and page number 871 of \cite{mil}) so that we can apply $p$-adic exponent pair $(1/30,13/15)$, when $p \neq 2, 3, 5$,   to the above inner sum to get 
\begin{align*}
T(R) &\ll_{p}  \left(\frac{p^{r-2 \kappa}}{R}\right)^{1/30} \, R^{13/15} N^{\epsilon} \\
&\ll_{p}  \, p^{\frac{r}{30}} \, R^{1/5} N^{\epsilon},
\end{align*}
where absolute constant depends on prime $p$.
Which concludes the lemma.
\end{proof}

The consequence of the above lemma we have
\begin{equation} \label{Delta}
\Delta(n, q_{i}, r_{1},N, \varepsilon) \ll \, p^{\frac{13r}{15}} \, N^{-5/6} N^{\epsilon}.
\end{equation}

In the following lemma we estimate $\Theta_{\text{non-zero}}$.
\begin{lemma}
We have 
$$\Theta_{\text{non-zero}} \ll \frac{ p^{\frac{28r}{15}+\ell -\frac{3\ell_{1}}{2}}}{N^{4/3}} N^{\epsilon}.$$
\end{lemma}

\begin{proof}
We have
$$\Theta_{\text{non-zero}} \ll \frac{p^{r+\ell -\frac{3\ell_{1}}{2}}}{\sqrt{N}} \sup_{\substack{q_{i} \leq Q \\ |r_{1}| \leq \frac{p^{r}}{N} \\ 0 < |n| \leq \frac{q_{1}q_{2} p^{\ell -\ell_{1}}} {N_{0}}}} |\Delta(n, q_{i}, r_{1},N, \varepsilon)| \, N^{\epsilon}.$$
By substituting the bound for $\Delta(...)$ from equation \eqref{Delta} in the above inequality we get
$$\Theta_{\text{non-zero}} \ll \frac{ p^{\frac{28r}{15}+\ell -\frac{3\ell_{1}}{2}}}{N^{4/3}} N^{\epsilon}.$$
\end{proof}

Let $T_{\neq 0}(\varepsilon, \ell_{1},N)$ and $S_{f,\chi}(N; r_{1}=0 \, \text{contri}, \Theta_{\text {non-zero}}) $ denote the contribution of $\Theta_{\text{non-zero}}$ to $T(\varepsilon, \ell_{1},N)$ and $S_{f,\chi}(N; r_{1}=0 \, \text{contri}) $ respectively. Then we have that 
$$T_{\neq 0}(\varepsilon, \ell_{1},N) \ll \frac{ p^{14r/15} p^{\frac{3 \ell -5\ell_{1}}{4}}}{N^{2/3}} N^{\epsilon}$$
and consequently we have that
\begin{align*}
S_{f,\chi}(N; r_{1}=0 \, \text{contri}, \Theta_{\text {non-zero}}) \ll \frac{ p^{13r/30} N^{7/12}}{p^{\ell /4}}\, N^{\epsilon}.
\end{align*}
Thus we have the following proposition.

\begin{proposition} \label{nonzerocontributiion}
We have 
$$S_{f,\chi}(N; r_{1}=0 \, \text{contri}, \Theta_{\text {non-zero}}) \ll \frac{ p^{13r/30} N^{7/12}}{p^{\ell /4}}\, N^{\epsilon}.$$
\end{proposition}

\section{conclusion}
In this section we complete the proof of our main Theorem \ref{mainth}. We have 
$$S_{f,\chi}(N; r_{1}=0 \, \text{contri}) = S_{f,\chi}(N; r_{1}=0 \, \text{contri}, \Theta_{\text {zero}}) + S_{f,\chi}(N; r_{1}=0 \, \text{contri}, \Theta_{\text {non-zero}}).$$
From propositions  \ref{zerocontributiion} and \ref{nonzerocontributiion}, we infer that
$$S_{f,\chi}(N; r_{1}=0 \, \text{contri}) \ll  \, \left(\sqrt{N} p^{\ell/2}+\frac{ p^{13r/30} N^{7/12}}{p^{\ell /4}}\right) N^{\epsilon},$$
provided $\max\{ p^{\ell} , p^{r-\ell}\} \leq N$, and $\ell \leq 2r/3$. By equating two terms in parenthesis we get the value of $\ell$ which is given by
$$ \ell = \frac{26r}{45} + \frac{1}{9} \log_{p}{N}.$$ 
Note that this choice of $\ell$ satisfies above conditions $p^{\ell} \neq $ and $\ell \leq 2r/3$  provided that $N\geq p^{13r/20}$ and $N\leq p^{4r/5}$ respectively. Therefore we conclude that 
$$S_{f, \chi}(N) \ll  N^{\frac{5}{9}} \ p^{\frac{13r}{45}} \, N^{\epsilon},$$
provided $p^{13r/20} \leq N\leq p^{4r/5}$ and absolute value may depend on prime $p$. This concludes the proof of Theorem \ref{mainth}.
 
\section{Acknowledgements}
Authors are thankful to Ritabrata Munshi, Satadal Ganguly, D. Surya Ramana, Sumit  Kumar, Saurabh Singh, and Prahlad for their support and encouragements. Authors are thankful to Djordje Milicevic for his useful comments on the article. The first author thanks the ISI,  Kolkata for the nice research environment. And the second author thanks the department of mathematics IIT Bomaby for excellent research facilities.

{}

\end{document}